\pdfoutput=1
\RequirePackage{etex}
\documentclass[a4paper,final]{amsart}

\usepackage[utf8]{inputenc}
\usepackage[T1]{fontenc}
\usepackage{libertine}
\usepackage[libertine]{newtxmath}
\usepackage[scr=rsfso]{mathalfa}
\usepackage{bm}

\usepackage{mathtools}
\usepackage{url}
\usepackage{hyperref}
\hypersetup{%
  draft=false,%
  pdfencoding=auto,%
  hypertexnames=false,%
  pdftitle={%
    On Convergence of Stringy Motives of wild p n cyclic quotient singularities},%
  pdfauthor={Mahito Tanno}%
}
\usepackage[capitalize]{cleveref}
\usepackage{autonum}
\crefformat{equation}{(#2#1#3)}
\crefrangeformat{equation}{(#3#1#4--#5#2#6)}
\crefrangeformat{align}{(#3#1#4--#5#2#6)}

\theoremstyle{plain}
\newtheorem{Theorem}{Theorem}[section]
\newtheorem{Proposition}[Theorem]{Proposition}
\newtheorem{Lemma}[Theorem]{Lemma}
\newtheorem{Corollary}[Theorem]{Corollary}
\theoremstyle{definition}
\newtheorem{Definition}[Theorem]{Definition}

\newtheorem{Remark}[Theorem]{Remark}

\usepackage{enumitem}

\newcommand{\setN}{\mathbb{N}}
\newcommand{\setQ}{\mathbb{Q}}
\newcommand{\setR}{\mathbb{R}}
\newcommand{\setZ}{\mathbb{Z}}
\newcommand{\setJ}{\mathcal{J}}
\newcommand{\setUrj}{\mathcal{U}} 
\newcommand{\setUtilde}{\tilde{\mathcal{U}}}

\newcommand{\rngMhat}{\hat{\mathcal{M}}}

\newcommand{\schA}{\mathbb{A}}
\newcommand{\schG}{\mathbb{G}}
\newcommand{\shfO}{\mathscr{O}}
\newcommand{\mtvL}{\mathbb{L}}  
\newcommand{\Mst}{\mathnormal{M}_{\text{\textnormal{st}}}} 

\DeclareMathOperator{\ord}{ord}

\DeclareMathOperator{\Img}{Im}

\newcommand{\floor}[1]{\left\lfloor#1\right\rfloor}
\newcommand{\ceil}[1]{\left\lceil#1\right\rceil}
\newcommand{\vfunc}{\boldsymbol{v}}     
\DeclareMathOperator{\Spec}{Spec}
\newcommand{\spcJ}[1][\infty]{\mathnormal{J}_{#1}}
\newcommand{\GCov}[1][G]{\operatorname{\mathnormal{#1}-Cov}}
\newcommand{\vect}[1]{\bm{#1}}
\newcommand{\modBdd}{\equiv_{\text{bdd}}} 

\title{%
  On convergence of stringy motives of wild \(p^{n}\)-cyclic quotient
  singularities
}

\author[M. Tanno]{Mahito Tanno}
\address{
  Department of Mathematics, Graduate School of Science, Osaka University,
  Toyonaka, Osaka 560-0043, Japan}
\email{\textsf{mahito@presche.me}}
\email{\textsf{u529757k@ecs.osaka-u.ac.jp}}


\begin{document}
\begin{abstract}
  The wild McKay correspondence, a variant of the McKay correspondence in
  positive characteristics, shows that stringy motives of quotient varieties
  equal some motivic integrals on the moduli space of of the Galois covers of a
  formal disk.  In this paper, we determine when the integrals converge for the
  the case of cyclic groups of prime power order.  As an application, we give a
  criterion for the quotient variety being canonical or log canonical.
\end{abstract}

\maketitle

\section{Introduction}\label{sec:introduction}
Let \(k\) be an algebraically closed field of characteristic \(p > 0\),
\(G = \setZ / p^{n} \setZ\) the cyclic group of order \(p^{n}\), and \(V\) a
\(G\)-representation of dimension \(d\).  We identify \(V\) with the
\(d\)-dimensional affine space \(\schA_{k}^{d}\).  Yasuda proved the wild McKay
correspondence theorem~\cite[{Corollary~16.3}]{Yasuda2019:Motivic}, which shows
that the stringy motive \(\Mst(V/G)\) of the quotient variety \(V/G\) equals the
integral of the form \(\int_{\GCov(D)} \mtvL^{d - \vfunc}\).  Here \(\GCov(D)\)
denotes the moduli space of \(G\)-covers of \(D = \Spec k[[t]]\) and \(\vfunc\)
the \(\vfunc\)-function associated to given \(G\)-representation \(V\).  Yasuda
and the author~\cite{Tanno2020:Wild} give an explicit formula for the
\(\vfunc\)-function, but it is not easy to compute the integrals in general.

The subject of this paper is the convergence of the integrals
\(\int_{\GCov(D)} \mtvL^{d - \vfunc}\).  When \(n \leq 2\), we already have
criteria \cite{Yasuda2019:Discrepancies,Tanno2020:Wild}.  We generalize them
to general \(n\) and then apply it to study singularities of the quotient
varieties \(V/G\).

The moduli space \(\GCov(D)\) can be described by the Artin--Schreier--Witt
theory and hence the integral can be written as a infinite series of the form
\begin{equation}
  \int_{\GCov(D)} \mtvL^{d - \vfunc} =
  \sum_{\vect{j}} [\GCov(D; \vect{j})] \mtvL^{d - \vfunc|_{\GCov(D; \vect{j})}},
\end{equation}
where \(\GCov(D) = \coprod_{\vect{j}} \GCov(D; \vect{j})\) is a stratification.
Moreover, the value \(\vfunc(E)\) of the \(\vfunc\)-function at
\(E \in \GCov(D; \vect{j})\) can be written in terms of (upper) ramification
jumps of the \(G\)-extension \(L/k((t))\) corresponding to \(E\).  By
considering the infinite series above as one of functions with upper
ramification jumps as variables, we see that the integral
\(\int_{\GCov(D)} \mtvL^{d - \vfunc}\) converges if and only if some linear
function tends to \(- \infty\).  For an indecomposable \(G\)-representation
\(V\) of dimension \(d\), we define the following invariants:
\begin{align}
  S_{d}^{(m)}
  &\coloneq \sum_{\substack{0 \leq i_{0} + i_{1} p + \dotsb i_{n -
          1} p^{n - 1} < d, \\ 0 \leq i_{0}, i_{1}, \dotsc, i_{n - 1} < p}}
  i_{m} \quad (0 \leq m \leq n - 1), \\
  D_{V}^{(m)}
  &\coloneq p^{n - 1} \left(
    S_{d}^{(m)} - (p - 1) \sum_{l = 0}^{m - 1} p^{m - 1} S_{d}^{(l)}\right)
    \quad (0 \leq m \leq n - 1).
\end{align}
We generalize them to decomposable ones in the way that they become additive for
direct sums.
\begin{Theorem}[\cref{thm:cond:converge}]
  Let \(V\) be a \(G\)-representation of dimension \(d\).  The integral
  \(\int_{\GCov(D)} \mtvL^{d - \vfunc_{V}}\) converges if and only if the
  inequalities
  \begin{equation}
    1 - \frac{1}{p^{n - m}}
    - \sum_{l = m}^{n - 1} \frac{D_{V}^{(l)}}{p^{2n - 1 - l}} < 0
    \quad (m = 0, 1, \dotsc, n - 1)
  \end{equation}
  hold.
\end{Theorem}

Using the wild McKay correspondence, we can study singularities of the quotient
\(V/G\); for instance, we have the following simple criterion.

\begin{Theorem}[\cref{prop:dimensional-criterion}]
  Assume that \(V\) is an effective indecomposable \(G\)-representation of
  dimension \(d\) which has no pseudo-reflection.  Let \(X = V/G\) be the
  quotient variety.  The following holds:
  \begin{enumerate}
  \item \(X\) is canonical if \(d \geq p - 1 + p^{n}\).  Furthermore, if there
    is a log resolution of \(X\), then the converse is also true.
  \item \(X\) is log canonical if and only if \(d \geq p - 1 + p^{n}\).
  \end{enumerate}
\end{Theorem}

We know when given \(G\)-representation has
pseudo-reflections~\cite[Lemma~4.6]{Tanno2020:Wild}.  Note that an effective
indecomposable \(G\)-representation of dimension \(d\) has no pseudo-reflection
if and only if \(d > 1 + p^{n - 1}\).

The outline of this paper is as follows.  In \cref{sec:preliminaries}, we review
basic facts for motivic integrals, the moduli space \(\GCov(D)\), the
\(\vfunc\)-functions, and singularities.  We discuss convergence of the integral
over connected \(G\)-covers in \cref{sec:over-connected-cover}, and then apply
it to the integral \(\int_{\GCov(D)} \mtvL^{d - \vfunc}\).  As application of
the main theorem, we give some criteria for whether the quotient variety \(V/G\)
is canonical (resp.\ log canonical) or not in
\cref{sec:indecomposable-case,sec:application}.

\subsection*{Notation and convention}
Unless otherwise noted, we follow the notation below.  We denote by \(k\) an
algebraically closed field of characteristic \(p > 0\), and by \(K = k((t))\)
the field of formal Laurent power series over \(k\).  We set \(G = \langle
\sigma \rangle\) a cyclic group of order \(p^{n}\).

\subsection*{Acknowledgments}
We would like to thank Takehiko Yasuda and Takahiro Yamamoto for their helpful
comments.  This work was supported by JSPS KAKENHI JP18H01112 and JP18K18710.

\section{Preliminaries}\label{sec:preliminaries}

\subsection{Motivic integration and stringy motives}
To state the wild McKay correspondence theorem, we briefly review motivic
integration and define stringy motives.

Let \(X\) be a \(k\)-variety \(X\) of dimension \(d\).  We denote by
\(\spcJ[n] X\) the space of \(n\)-jets and by \(\spcJ X\) the space of arcs.
The motivic measure \(\mu_{X}\) on \(\spcJ X\) takes values in the ring
\(\rngMhat'\), which is a version of the completed Grothendieck ring of
varieties (see~\cite{Yasuda2014:p-Cyclic} for details).  The element of
\(\rngMhat'\) defined by \(Y\) is denoted by \([Y]\).  We write
\(\mtvL = [\schA_{k}^{1}]\).  Let \(\pi_{n} \colon \spcJ X \to \spcJ[n] X\) be
the truncation map.  We call a subset \(C \subset \spcJ X\) \emph{stable} if
there exists \(n \in \setN\) such that \(\pi_{n}(C) \subset \spcJ[n] X\) is
constructible, \(C = \pi_{n}^{-1}(\pi_{n} C)\), and the map
\(\pi_{m + 1}(C) \to \pi_{m}(C)\) is a piecewise trivial
\(\schA_{k}^{d}\)-bundle for every \(m \geq n\).  We define the measure
\(\mu_{X}(C)\) of a stable subset \(C \subset \spcJ X\) by
\(\mu_{X}(C) \coloneq [\pi_{n}(C)] \mtvL^{-nd}\) for \(n \gg 0\).  For a more
general measurable subset, we define its measure as the limit of ones of stable
subsets.  For a measurable subset \(C \subset \spcJ X\) and a function
\(F \colon C \to \setZ \cup \{\infty\}\) such that every fiber is constructible,
we define
\(\int_{C} \mtvL^{F} \coloneq \sum_{m \in \setZ} \mu_{X}(F^{-1}(m)) \mtvL^{m}\).

We assume that \(X\) is normal and its canonical sheaf \(\omega_{X}\) is
invertible.  We then define the \emph{\(\omega\)-Jacobian ideal}
\(\mathcal{J}_{X} \subset \shfO_{X}\) by
\(\mathcal{J}_{X} \omega_{X} = \Img {\left( \bigwedge^{d} \Omega_{X/k} \to
    \omega_{X}\right)}\) and the \emph{stringy motive} \(\Mst(X)\) of \(X\) by
\begin{equation}
  \Mst(X) \coloneq \int_{\spcJ X} \mtvL^{\ord \mathcal{J}_{X}}.
\end{equation}
Here \(\ord\) denotes the order function associated to an ideal sheaf.

\begin{Remark}
  In our situation where \(G = \setZ/p^{n}\setZ\) acts on \(\schA_{k}^{d}\)
  linearly, the quotient variety \(X = \schA_{k}^{d} / G\) is \(1\)-Gorenstein,
  that is, the canonical sheaf \(\omega_{X}\) is invertible (see
  \cite[{Theorem~3.1.8}]{Campbell2011:Modular}).
\end{Remark}

\subsection{The wild McKay correspondence}
Yasuda proved the following.

\begin{Theorem}[{%
  the wild McKay correspondence~\cite[Corollary~16.3]{Yasuda2019:Motivic}}]
  Let \(G\) be an arbitrary finite group.  Assume that \(G\) acts on
  \(\schA_{k}^{d}\) linearly and effectively and that \(G\) has no
  pseudo-reflection.  Then we have
  \begin{equation}
    \Mst(\schA_{k}^{d} / G) = \int_{\GCov(D)} \mtvL^{d - \vfunc}.
  \end{equation}
  Here \(\GCov(D)\) denotes the moduli space of \(G\)-covers of
  \(D = \Spec k[[t]]\), and \(\vfunc\) is the \(\vfunc\)-function
  \(\vfunc \colon \GCov(D) \to \setQ\) associated to the \(G\)-action on
  \(\schA_{k}^{d}\).
\end{Theorem}

By a \(G\)-cover \(E\) of \(D\), we mean the normalization of \(D\) in an
\'{e}tale \(G\)-cover \(E^{*}\) of \(D^{*} = \Spec K\).  Note that the
\(\vfunc\)-function depends on the given \(G\)-representation.  We sometimes
write the \(\vfunc\)-function as \(\vfunc_{V}\), referring to the
\(G\)-representation \(V\) in question.

Let us consider the case \(G = \setZ/p^{n}\setZ\), which is of our principal
interest.

We can describe the moduli space \(\GCov(D)\) by using the Artin--Schreier--Witt
theory; There is a one-to-one correspondence between \(G\)-covers \(E\) and
\emph{reduced} Witt vectors \(\vect{f} = (f_{0}, f_{1}, \dotsc, f_{n - 1})\)
(\(f_{m} \in k[t^{-1}]\)).  Here \emph{reduced} means each \(f_{m}\) is of the
form \(f_{m} = \sum_{j > 0, p \nmid j} a_{m, j} t^{-j}\) (\(a_{m, j} \in k\)).
Moreover, we stratify \(\GCov(D)\) as follows:
\begin{align}
  \GCov(D) &= \coprod_{\vect{j}} \GCov(D; \vect{j}), \\
  \GCov(D; \vect{j}) &\leftrightarrow
                       \{\vect{f} = (f_{0}, f_{1}, \dotsc, f_{n - 1}) \mid
                       \ord f_{m} = - j_{m} (m = 0, 1, \dotsc, n - 1)\},
\end{align}
where \(\vect{j} = (j_{0}, j_{1}, \dotsc, j_{n - 1})\) is an \(n\)-tuple of
positive integers with \(p \nmid j_{m}\) or \(- \infty\).  Note that we can
identity \(\GCov(D; \vect{j})\) with
\(\prod_{j_{m} \neq - \infty} \schG_{k, m} \times_{k} \schA_{k}^{j_{m} - 1 %
  - \floor{j_{m}/p}}\).  See \cite[Section~2]{Tanno2020:Wild} for details.

For the definition of \(\vfunc\)-function, see
\cite[Definition~5.4]{Yasuda2016:Wilder}.  In our case
\(G = \setZ / p^{n} \setZ\), we can compute the value \(\vfunc_{V}(E)\) as
follows:

\begin{Theorem}[{\cite[Theorem~3.3]{Tanno2020:Wild}}]\label{thm:v-func}
  Let \(E\) be a connected \(G\)-cover and \(L/K\) the corresponding
  \(G\)-extension.  Assume that the \(G\)-extension \(L/K\) is defined by an
  equation \(\wp(\vect{g}) = \vect{f}\), where \(\vect{f}\) is a reduced Witt
  vector with \(- j_{m} = \ord f_{m}\) (\(j_{m} \in \setN' \cup \{- \infty\}\)).
  Note that \(j_{0} \neq - \infty\) since \(E\) is connected.  Put
  \begin{align}
    u_{i} &= \max \{p^{i - 1 - m} j_{m} \mid m = 0, 1, \dotsc, i - 1\}, \\
    l_{i} &= u_{0} + (u_{1} - u_{0})p + \dotsb + (u_{i} - u_{i - 1}) p^{i}.
  \end{align}
  Then
  \begin{equation}
    \vfunc_{V}(E)
    = \sum_{\substack{0 \leq i_{0} + pi_{1} + \dotsb + p^{n-1}i_{n-1} < d, \\
                      0 \leq i_{0} , i_{1}, \dotsc, i_{n-1} < p}} %
        \ceil{\frac{
                    i_{0}p^{n-1}l_{0} + i_{1}p^{n-2}l_{1} + \dotsb +
                    i_{n-1}l_{n-1}
                  }{p^{n}}}.
  \end{equation}
\end{Theorem}

When \(E\) is not connected with a connected component \(E'\) and the stabilizer
subgroup \(G' \subset G\), then we have
\begin{equation}
  \vfunc_{V}(E) = \vfunc_{V'}(E'),
\end{equation}
where \(V'\) is the restriction of \(V\) to \(G'\).

\begin{Remark}
  In the situation of \cref{thm:v-func}, we remark that
  \begin{enumerate}
  \item \(u_{i}\) (resp.\ \(l_{i}\)) are the upper (resp.\ lower) ramification
    jumps of the extension \(L/K\),
  \item the function \(\vfunc_{V}\) is constant on each stratum
    \(\GCov(D; \vect{j})\).
  \end{enumerate}
\end{Remark}

\subsection{Singularities}
We can study singularities of quotient varieties via the wild McKay
correspondence.

\begin{Proposition}[%
  {\cite[Proposition~6.6]{Yasuda2014:p-Cyclic}},
  {\cite[Corollary~16.4]{Yasuda2019:Motivic}}]\label{thm:Yasuda2014:Prop6.6}
  Let \(X = \schA_{k}^{d} / G\) be the quotient by a finite group \(G\).  If the
  integral \(\int_{\GCov(D)} \mtvL^{d - \vfunc}\) converges, then \(X\) is
  canonical.  Furthermore, if there is a log resolution of \(X\), then the
  converse is also true.
\end{Proposition}

Furthermore, we can also estimate discrepancies of quotient varieties
\(X = \schA_{k}^{d} / G\) by computing the integral
\(\int_{\GCov(D)} \mtvL^{d - \vfunc}\).  See \cite{Yasuda2019:Discrepancies} for
details (see also \cite[{Section~4.4}]{Tanno2020:Wild} for the case
\(G = \setZ / p^{n} \setZ\)).

\section{Integrals over connected
  \texorpdfstring{\(G\)}{G}-covers}\label{sec:over-connected-cover}
In this section, we consider integrals over connected \(G\)-covers,
\(\int_{\GCov^{0}(D)} \mtvL^{d - \vfunc_{V}}\), where
\(\GCov^{0}(D) = \coprod_{j_{0} \neq - \infty} \GCov(D; \vect{j})\) denotes the
set of connected \(G\)-covers of the formal disk \(D = \Spec k[[t]]\).  As we
see in \cref{thm:v-func}, for a connected \(G\)-cover \(E\), the value
\(\vfunc_{V}(E)\) is determined by the upper ramification jumps of the
corresponding \(G\)-extension \(L/K\).  By abuse of notation, let us consider
\(\vfunc_{V}\) as a function in variables
\(\vect{u} = (u_{0}, u_{1}, \dotsc, u_{n - 1})\).  The following is well-known
(see, for example, \cite[Lemma~3.5]{Obus2010:Wild}).

\begin{Lemma}\label{prop:cond:upper-ram-jumps}
  Let \(\vect{u} = (u_{0}, u_{1}, \dotsc, u_{n - 1})\) be an increasing sequence
  of positive integers. Then \(\vect{u}\) occurs as the set of upper
  ramification jumps of a \(G\)-extension of \(K\) if and only if the following
  conditions hold:
  \begin{enumerate}
  \item \(p \nmid u_{0}\), and
  \item for \(1 \leq i \leq n - 1\), either
    \begin{enumerate}[label=(2.\alph*)] 
    \item \(u_{i} = p u_{i - 1}\) or
    \item both \(u_{i} > p u_{i - 1}\) and \(p \nmid u_{i}\).
    \end{enumerate}
  \end{enumerate}
\end{Lemma}

We denote by \(\setUrj\) the set of increasing sequences of positive integers
satisfying the conditions of \cref{prop:cond:upper-ram-jumps}.  For \(\vect{u} =
(u_{0}, u_{1}, \dotsc, u_{n - 1})\), set
\begin{equation}
  \setJ(\vect{u}) \coloneq \left\{ \vect{j} = (j_{0}, j_{1}, \dotsc, j_{n - 1})
    \mathrel{} \middle| \mathrel{}
    u_{m} = \max\{ p^{m - 1 - i} j_{i} \mid i = 0, 1, \dotsc, m - 1 \} \right\}.
\end{equation}
Then we obtain
\begin{equation}\label{eq:integral/connected-covers}
  \int_{\GCov^{0}(D)} \mtvL^{d - \vfunc_{V}}
  = \sum_{\vect{u} \in \setUrj} \left(
      \sum_{\vect{j} \in \setJ(\vect{u})} [\GCov(D; \vect{j})] \right)
      \mtvL^{d - \vfunc_{V}(\vect{u})}.
\end{equation}
In addition, by definition, we have
\begin{equation}
  \dim \sum_{\vect{j} \in \setJ(\vect{j})} [\GCov(D; \vect{j})]
  = d + u_{0} - \floor{u_{0}/p} + u_{1} - \floor{u_{1}/p} + \dotsb
    + u_{n - 1} - \floor{u_{n - 1} / p}.
\end{equation}
Therefore, it is enough to study the asymptotic behavior of the function
\begin{equation}\label{cond:pre:criterion}
  u_{0} - \floor{u_{0}/p} + u_{1} - \floor{u_{1}/p} + \dotsb
  + u_{n - 1} - \floor{u_{n - 1} / p}
  - \vfunc_{V}(u_{0}, u_{1}, \dotsc, u_{n - 1})
\end{equation}
in variables \(u_{0}, u_{1}, \dotsc, u_{n - 1}\).

Let us define some invariants to study the function \(\vfunc_{V}\).

\begin{Definition}
  For a positive integer \(d\) (\(d \leq p^{n}\)) and a non-negative integer
  \(m\) (\(m \leq n - 1\)), we define
  \begin{equation}
    S_{d}^{(m)} \coloneq \sum_{\substack{0 \leq i_{0} + i_{1} p + \dotsb i_{n -
          1} p^{n - 1} < d, \\ 0 \leq i_{0}, i_{1}, \dotsc, i_{n - 1} < p}}
    i_{m}.
  \end{equation}
\end{Definition}

Namely, \(S_{d}^{(m)}\) is the sum of the \((m + 1)\)-th digits of the integers
\(0, 1, \dotsc, d - 1\) in base-\(p\) notation.  We can write them explicitly.

\begin{Lemma}
  Put \(d = d_{0} + d_{1} p + \dotsb + d_{n - 1} p^{n - 1}\)
  (\(0 \leq d_{m} < p\) for \(m = 0, 1, \dotsc, n - 2\); \(0 \leq d_{n - 1}\)).
  Then the equality
  \begin{equation}
    S_{d}^{(m)} = p^{m} S_{d_{m} + \dotsb + d_{n - 1} p^{n - 1 - m}}^{(0)}
                  + \sum_{l = 0}^{m - 1} p^{l} d_{l} d_{m} \quad (m > 0)
  \end{equation}
  holds.  In addition, we have
  \begin{equation}
    S_{d}^{(0)} = (d_{1} + d_{2} p + \dotsb d_{n - 1} p^{n - 2})
      \frac{p (p - 1)}{2} + \frac{d_{0} (d_{0} - 1)}{2}.
  \end{equation}
\end{Lemma}
\begin{proof}
  The second equality is obvious.  Let
  \(q = d_{1} + d_{2} p + \dotsb + d_{n - 1} p^{n - 2}\).  By definition, we
  have
  \begin{align}
    \label{eq:S_d^(m):induction}
    S_{d}^{(m)} =
    \sum_{\substack{
            0 \leq i_{0} + i_{1} p + \dotsb i_{n - 1} p^{n - 1} < d, \\
            0 \leq i_{0}, i_{1}, \dotsc, i_{n - 1} < p}} i_{m}
    &= \sum_{i_{0} = 0}^{p - 1}
      \sum_{\substack{
              0 \leq i_{1} + i_{2} p + \dotsb i_{n - 1} p^{n - 2} < q, \\
              0 \leq i_{1}, i_{1}, \dotsc, i_{n - 1} < p}} i_{m}
    + \sum_{i_{0} = 0}^{d_{0} - 1} d_{m} \\
    &= p \sum_{\substack{
                0 \leq i_{0} + i_{1} p + \dotsb i_{n - 2} p^{n - 2} < q, \\
                0 \leq i_{0}, i_{1}, \dotsc, i_{n - 2} < p}} i_{m - 1}
      + d_{0} d_{m} \\
    &= p S_{q}^{(m - 1)} + d_{0} d_{m}.
  \end{align}
  Hence we obtain the first equality by induction.
\end{proof}

\begin{Definition}
  Let \(V\) be an indecomposable \(G\)-representation of dimension \(d\).  We
  define
  \begin{equation}
    D_{V}^{(m)} \coloneq p^{n - 1} {\left( S_{d}^{(m)}
        - (p - 1) \sum_{l = m + 1}^{n - 1} p^{m - l} S_{d}^{(l)} \right)}.
  \end{equation}
  For decomposable \(G\)-representations, we define the invariants
  \(D_{V}^{(m)}\) in the way that they become additive for direct sums.
\end{Definition}

\begin{Lemma}\label{prop:reminder-vfunc}
  For integers \(q_{m}\) and \(r_{m}\) (\(m = 0, 1, \dotsc, n - 1\)), we have
  \begin{equation}
    \vfunc_{V}(q_{0} p^{n} + r_{0}, q_{1} p^{n} + r_{1}, \dotsc,
      q_{n - 1} p^{n} + r_{n -1})
    = \sum_{m = 0}^{n - 1} D_{V}^{(m)} q_{m}
      + \vfunc_{V}(r_{0}, r_{1}, \dots, r_{n - 1}).
  \end{equation}
\end{Lemma}
\begin{proof}
  Since the function \(\vfunc_{V}\) and the invariants \(D_{V}^{(m)}\) are
  additive with respect to direct sums of \(G\)-representations, we may assume
  that \(V\) is indecomposable of dimension \(d\).

  By direct computing with \cref{thm:v-func}, we obtain
  \begin{align}
    & \vfunc_{V}(q_{0} p^{n} + r_{0}, u_{1} p^{n} + r_{1}, \dotsc,
      q_{n - 1} p^{n} + r_{n - 1}) \\
    &= \sum_{\substack{
      0 \leq i_{0} + i_{2} p + \dotsb i_{n - 1} p^{n - 1} < d, \\
      0 \leq i_{0}, i_{1}, \dotsc, i_{n - 1} < p}}
      \sum_{m = 0}^{n - 1} i_{m} p^{n - 1 - m}
      (q_{0} + (q_{1} - q_{0}) p + \dotsb + (q_{m} - q_{m - 1}) p^{m}) \\
    &\qquad\qquad + \vfunc_{V}(r_{0}, r_{1}, \dotsc, r_{n - 1}).
  \end{align}
  Moreover, we have
  \begin{align}
    & \sum_{m = 0}^{n - 1} i_{m} p^{n - 1 - m}
      (q_{0} + (q_{1} - q_{0}) p + \dotsb + (q_{m} - q_{m - 1}) p^{m}) \\
    =& p^{n - 1} i_{0} q_{0}
       + p^{n - 2} i_{1} (- (p - 1) q_{0} + p q_{1}) + \dotsb \\
    &+ i_{n - 1} (- (p - 1) q_{0} - p (p - 1) q_{1} - \dotsb
      - p^{n -2} (p - 1) q_{n - 2} + p^{n - 1} q_{n - 1}) \\
    =& p^{n - 1} (i_{0} - (p - 1)
      (p^{-1} i_{1} + \dotsb + p^{- n + 1} i_{n - 1} )) q_{0} \\
    &+ p^{n - 1} (i_{1} - (p - 1)
      (p^{-1} i_{2} + \dotsb + p^{- n + 2} i_{n - 1})) q_{1} + \dotsb \\
    &+ p^{n - 1} i_{n - 1} q_{n - 1}.
  \end{align}
  The above equalities together with the definition of \(D_{V}^{(m)}\) show the
  lemma.
\end{proof}

We state the following as a conclusion of this section.

\begin{Theorem}\label{prop:integral/connected-covers}
  Let \(V\) be a \(G\)-representation of dimension \(d\) (not necessarily
  indecomposable).  The integral \(\int_{\GCov^{0}(D)} \mtvL^{d - \vfunc_{V}}\)
  on the space \(\GCov^{0}(D)\) of the connected \(G\)-covers converges if and
  only if the strict inequalities
  \begin{equation}
    1 - \frac{1}{p^{n - m}}
    - \sum_{l = m}^{n - 1} \frac{D_{V}^{(l)}}{p^{2 n - 1 -l}} < 0
    \quad (m = 0, 1, \dots, n  -1)
  \end{equation}
  hold.  If the inequalities \(\leq 0\) hold, then the integral
  \(\int_{\GCov^{0}(D)} \mtvL^{d - \vfunc}\) has terms of dimension bounded
  above.
\end{Theorem}
\begin{proof}
  It is obvious that the integral \(\int_{\GCov^{0}(D)} \mtvL^{d - \vfunc_{V}}\)
  converges if and only if \cref{cond:pre:criterion} tends to \(- \infty\) as
  the all variables \(u_{m}\) increase.

  From \cref{prop:reminder-vfunc}, we have
  \begin{align}
    & u_{0} - \floor{u_{0}/p} + u_{1} - \floor{u_{1}/p} + \dotsb
      + u_{n - 1} - \floor{u_{n - 1} / p}
      - \vfunc_{V}(u_{0}, u_{1}, \dotsc, u_{n - 1}) \\
    &\modBdd
      u_{0} - u_{0}/p + u_{1} - u_{1}/p + \dotsb + u_{n - 1} - u_{n - 1}/p
      - \sum_{m = 0}^{n - 1} D_{V}^{(m)} u_{m}/p^{n} \\
    &=
      \sum_{m = 0}^{n - 1} \left(
      1 - \frac{1}{p} - \frac{D_{V}^{(m)}}{p^{n}}\right) u_{m}
      \eqqcolon f(\vect{u}),
  \end{align}
  where \(\modBdd\) means equivalence modulo bounded functions.  What we want to
  study is the limit of the function \(f(\vect{u})\).  Thus we consider
  \(f(\vect{u})\) as a function defined on
  \begin{equation}
    \setUtilde \coloneq \{
    \vect{u} = (u_{0}, u_{1}, \dotsc, u_{n - 1}) \in \setR^{n}
    \mid u_{0} \geq 1, u_{i} \geq p u_{i - 1} \, (i = 1, 2, \dotsc, n - 1)\}
  \end{equation}
  instead of \(\setUrj\).  For \(t \in \setR_{\geq 1}\), let
  \(\setUtilde_{t} \coloneq \setUtilde \cap \{u_{n - 1} = t\}\) be the
  intersection of the polyhedron \(\setUtilde\) and the hyperplane
  \(u_{n - 1} = t\).  Assume \(t \geq p^{n - 1}\) so that \(\setUtilde_{t}\)
  becomes non-empty.  Obviously, \(\setUtilde_{t}\) is bounded, that is, it is a
  polytope.  Since \(f\) is a linear function, thus the maximum value of
  \(f|_{\setUtilde_{t}}\) is attained at the one of its vertices
  \begin{equation}
    (1, p \dotsc, p^{n - 2}, t),
    (1, p \dotsc, p^{n - 3}, t/p, t), \dotsc,
    (t/p^{n - 1}, \dotsc, t/p, t) \in \setUtilde_{t}.
  \end{equation}
  By substituting, we have
  \begin{align}
    f(\overbrace{1, p \dotsc, p^{m - 1}}^{m}, t/p^{n - 1 -m}, \dotsc, t/p, t)
    &\modBdd \sum_{l = m}^{n - 1} \left(1 - \frac{1}{p} -
      \frac{D_{V}^{(l)}}{p^{n}}\right) \frac{t}{p^{n - 1 - l}} \\
    &= \left(
      1 - \frac{1}{p^{n - m}}
      - \sum_{l = m}^{n - 1} \frac{D_{V}^{(l)}}{p^{2n - 1 - l}}\right) t.
  \end{align}

  Therefore, this shows that the function \(f(\vect{u})\) tends to \(- \infty\)
  if and only if the coefficients
  \(1 - 1/p^{n - m} - \sum_{l = m}^{n - 1} D_{V}^{(l)}/p^{2n - 1 - l}\) are all
  negative.
\end{proof}

\section{Convergence of stringy motives}\label{sec:convergence}
In this section, we prove our main theorem.

\begin{Theorem}\label{thm:cond:converge}
  Let \(V\) be an effective \(G\)-representation of dimension \(d\).  The
  integral \(\int_{\GCov(D)} \mtvL^{d - \vfunc_{V}}\) converges if and only if
  the strict inequalities
  \begin{equation}
    1 - \frac{1}{p^{n - m}}
    - \sum_{l = m}^{n - 1} \frac{D_{V}^{(l)}}{p^{2n - 1 - l}} < 0
    \quad (m = 0, 1, \dotsc, n - 1)
  \end{equation}
  hold.  If the inequalities \(\leq 0\) hold, then the integral
  \(\int_{\GCov(D)} \mtvL^{d - \vfunc}\) has terms of dimension bounded above.
\end{Theorem}
\begin{proof}
  We prove by induction on \(n\).  The case \(n = 1\) is just
  \cite[Proposition~6.9]{Yasuda2014:p-Cyclic}.  Let \(H = \setZ/p^{n - 1}\setZ\)
  be the subgroup of \(G\) of index \(p\) and \(W\) the restriction of \(V\) to
  \(H\).  Let us divide the integral as follows:
  \begin{equation}
    \int_{\GCov(D)} \mtvL^{d - \vfunc_{V}}
    = \int_{\GCov[H](D)} \mtvL^{d - \vfunc_{W}}
    + \int_{\GCov^{0}(D)} \mtvL^{d - \vfunc_{V}}.
  \end{equation}
  Note that the necessary and sufficient condition on convergence of
  \(\int_{\GCov[H](D)} \mtvL^{d - \vfunc_{W}}\) is given by the induction
  hypothesis, and one of \(\int_{\GCov^{0}(D)} \mtvL^{d - \vfunc_{V}}\) is by
  \cref{prop:integral/connected-covers}.  From the lemma below, we have
  \begin{align}
    1 - \frac{1}{p^{n - m}}
    - \sum_{l = m}^{n - 1} \frac{D_{V}^{(l)}}{p^{2n - 1 -l}}
    &= 1 - \frac{1}{p^{n - m}}
      - \sum_{l = m}^{n - 1} \frac{p D_{W}^{(l - 1)}}{p^{2n - 1 -l}} \\
    &= 1 - \frac{1}{p^{(n - 1) - (m - 1)}}
      - \sum_{l = m - 1}^{n - 2} \frac{D_{W}^{(l)}}{p^{2 (n - 1) - 1 - l}}
  \end{align}
  for \(m = 1, 2, \dotsc, n - 1\).  Therefore, the convergence of the integral
  \(\int_{\GCov^{0}(D)} \mtvL^{d - \vfunc_{V}}\) implies that of the integral
  \(\int_{\GCov[H](D)} \mtvL^{d - \vfunc_{W}}\), and hence the proof is
  completed.
\end{proof}

\begin{Lemma}
  In the situation of \cref{thm:cond:converge}, for
  \(m = 1, 2, \dotsc, n - 1 \), we have \(D_{V}^{(m)} = p D_{W}^{(m - 1)}\).
\end{Lemma}
\begin{proof}
  Since the invariants \(D_{V}^{(m)}\) are additive with respect to direct sum,
  thus we may assume that \(V\) is indecomposable.  Moreover, by
  \cite[Lemma~4.5]{Tanno2020:Wild}, we obtain
  \begin{equation}
    W \simeq W_{q + 1}^{\oplus r} \oplus W_{q}^{\oplus p - r},
  \end{equation}
  where \(d = r + qp\) (\(0 \leq r < p\)) and \(W_{e}\) denotes the
  indecomposable \(G\)-representation of dimension \(e\).  Put
  \(d = d_{0} + d_{1} p + \dotsb d_{n - 1} p^{n - 1}\) (\(0 \leq d_{m} < p\) for
  \(m = 0, 1, \dotsc, n - 2\)).  Note that \(r = d_{0}\) and
  \(q = d_{1} + d_{2} p + \dotsc + d_{n - 1} p^{n - 2}\).  By definition, we
  have
  \begin{align}
    S_{q + 1}^{(m - 1)}
    &= \sum_{\substack{
      0 \leq i_{0} + i_{1} p + \dotsb + i_{n - 2} p^{n - 2} < q + 1, \\
      0 \leq i_{0}, i_{1}, \dotsc, i_{n - 2} < p }} i_{m - 1} \\
    &= \sum_{\substack{
      0 \leq i_{0} + i_{1} p + \dotsb + i_{n - 2} p^{n - 2} < q, \\
      0 \leq i_{0}, i_{1}, \dotsc, i_{n - 2} < p }} i_{m - 1}
      + \sum_{\substack{
      0 \leq i_{0} + i_{1} p + \dotsb + i_{n - 2} p^{n - 2} = q + 1, \\
      0 \leq i_{0}, i_{1}, \dotsc, i_{n - 2} < p }} i_{m - 1} \\
    &= \sum_{\substack{
      0 \leq i_{0} + i_{1} p + \dotsb + i_{n - 2} p^{n - 2} < q, \\
      0 \leq i_{0}, i_{1}, \dotsc, i_{n - 2} < p }} i_{m - 1} + d_{m} \\
    &= S_{q}^{(m - 1)} + d_{m},
  \end{align}
  and hence
  \begin{equation}
    r S_{q + 1}^{(m - 1)} + (p - r) S_{q}^{(m - 1)}
    = p S_{q}^{(m - 1)} + r d_{m}.
  \end{equation}
  Therefore, combining \cref{eq:S_d^(m):induction}, we obtain
  \begin{equation}
    S_{d}^{(m)} = r S_{q + 1}^{(m - 1)} + (p - r) S_{q}^{(m - 1)}.
  \end{equation}
  Now the lemma follows from the definition of \(D_{V}^{(m)}\).
\end{proof}

\begin{Corollary}\label{prop:criterion:canonical}
  Let \(X \coloneq V/G\) be the quotient variety.
  \begin{enumerate}
  \item \(X\) is canonical if the strict inequalities
    \(1 - 1/p^{n - m} - \sum_{l = m}^{n - 1} D_{V}^{(l)}/p^{2n - 1 - l} < 0\)
    (\(m = 0, 1, \dotsc, n - 1\)) hold.  Furthermore, if there is a log
    resolution of \(X\), then the converse is also true.
  \item \(X\) is log canonical if and only if the inequalities
    \(1 - 1/p^{n - m} - \sum_{l = m}^{n - 1} D_{V}^{(l)}/p^{2n - 1 - l} \leq 0\)
    (\(m = 0, 1, \dotsc, n - 1\)) hold.
  \end{enumerate}
\end{Corollary}
\begin{proof}
  (1).  If the strict inequalities hold, then the integral
  \(\int_{\GCov(D)} \mtvL^{d - \vfunc}\) and hence the stringy motive
  \(\Mst(X)\) converges.  From \cite[Proposition~6.6]{Yasuda2014:p-Cyclic}, we
  obtain the claim.

  (2).  Holding the inequalities is equivalent
  to that the integral \(\int_{\GCov(D)} \mtvL^{d - \vfunc_{V}}\) has terms of
  dimensions bounded above.  Hence, from
  \cite[Corollary~16.4~(1)]{Yasuda2019:Motivic}, we obtain desired conclusion.
\end{proof}

\section{Indecomposable cases}\label{sec:indecomposable-case}
We give more precise estimation for the indecomposable cases.

\begin{Theorem}\label{prop:dimensional-criterion}
  Assume that \(V\) is an indecomposable \(G\)-representation of dimension \(d\)
  which has no pseudo-reflection.  Let \(X \coloneq V/G\) be the quotient
  variety.
  \begin{enumerate}
  \item \(X\) is canonical if \(d \geq p + p^{n - 1}\).  Furthermore, if there
    is a log resolution of \(X\), then the converse is also true.
  \item \(X\) is log canonical if and only if \(d \geq p - 1 + p^{n - 1}\).
  \end{enumerate}
\end{Theorem}

\begin{Lemma}\label{prop:D_V:increasing-func}
  We consider the invariants \(D_{V}^{(m)}\) as functions in variable \(d\).
  Then the sum \(\sum_{l = m}^{n - 1} D_{V}^{(l)} / p^{2n - 1 -l}\) is strictly
  increasing.
\end{Lemma}
\begin{proof}
  By definition, we have
  \begin{align}
    \sum_{l = m}^{n - 1} p^{l} D_{V}^{(l)}
    &= \sum_{l = m}^{n - 1} p^{l} \left(
      - (p - 1) \sum_{j = l + 1}^{n - 1} p^{l - j} S_{d}^{(j)}
      + S_{d}^{(l)}\right) \\
    &= p^{n - 1} \Big( p^{m} S_{d}^{(m)}
    + (- (p - 1) p^{m - (m + 1)} + p^{m + 1}) S_{d}^{(m + 1)} \\
    &\quad + \dotsb \\
    &\quad + (- (p - 1) (p^{m - (n - 1)} + p^{(m + 1) - (n - 1)} + \dotsb
      + p^{(n - 2) - (n - 1)}) + p^{n - 1}) S_{d}^{(n - 1)} \Big) \\
    &= p^{n - 1} \sum_{l = m}^{n - 1} \left(
      - (p - 1) \sum_{j = m}^{l - 1} p^{j - l} + p^{l} \right) S_{d}^{(l)} \\
    &= p^{n - 1} \sum_{l = m}^{n - 1} (p^{m - l} - 1 + p^{l}) S_{d}^{(l)} \\
    &\geq p^{n - 1} \cdot p^{n - 1} S_{d}^{(n - 1)} = p^{n - 1} D_{V}^{(n - 1)}.
  \end{align}
  Since \(D_{V}^{(n - 1)} = p^{n - 1} S_{d}^{(n - 1)}\) is strictly increasing
  with respect to \(d\), thus we get the desired conclusion.
\end{proof}

\begin{Remark}
  From the proof above, \(\sum_{l = m}^{n - 1} p^{l} D_{V}^{(l)}\) is monotone
  decreasing with respect to \(m\).  Since \(D_{V}^{(n - 1)}\) is non-negative,
  thus \(D_{V}^{(l)}\) are all non-negative.
\end{Remark}

\begin{proof}[Proof of \cref{prop:dimensional-criterion}]
  The cases \(n = 1\) and \(n = 2\) are proved in
  \cite{Yasuda2019:Discrepancies} and \cite{Tanno2020:Wild} respectively.  We
  assume that \(n \geq 3\).  It is enough to show that the inequalities in
  \cref{prop:criterion:canonical} hold.

  (2).  We consider the case \(d = p -1 + p^{n - 1}\).  By direct computation,
  we obtain
  \begin{align}
    S_{d}^{(0)} &= \frac{p^{n - 1} (p - 1)}{2} + \frac{(p - 1) (p - 2)}{2}, \\
    S_{d}^{(m)} &=
                  \begin{cases}
                    p^{n - 1} (p - 1)/2 & \text{if \(0 < m < n - 1\)}, \\
                    p - 1               & \text{if \(m = n - 1\)},
                  \end{cases}
  \end{align}
  and hence
  \begin{align}
    D_{V}^{(n - 1)} &= p^{n - 1} (p - 1), \\
    D_{V}^{(n - 2)} &= p^{n - 1} (p - 1) \left(
                      \frac{p^{n - 1}}{2} - \frac{p - 1}{p}\right).
  \end{align}
  For \(m = n - 1\), we have
  \begin{equation}
    1 - \frac{1}{p^{n - (n - 1)}} - \frac{D_{V}^{(n - 1)}}{p^{2n - 1 - (n - 1)}}
    = 1 - \frac{1}{p} - \frac{p - 1}{p}
    = 0.
  \end{equation}
  Since \(D_{V}^{(m)}\) are strictly increasing with respect to \(d\), from
  \cref{prop:criterion:canonical} (2), thus \(X\) is not log canonical when
  \(d < p - 1 + p^{n - 1}\).

  On the other hand, for \(m = 0, 1, \dotsc, n - 2\), we have
  \begin{align}
    1 - \frac{1}{p^{n - m}}
    - \sum_{l = m}^{n - 1} \frac{D_{V}^{(l)}}{p^{2n - 1 -l}}
    &< 1 - \sum_{l = m}^{n - 1} \frac{D_{V}^{(l)}}{p^{2n - 1 -l}} \\
    &\leq 1 - \left(
      \frac{D_{V}^{(n - 2)}}{p^{n + 1}} + \frac{D_{V}^{(n)}}{p^{n}}
      \right).
  \end{align}
  Thus we obtain
  \begin{align}
    1 - \left(
      \frac{D_{V}^{(n - 2)}}{p^{n + 1}} + \frac{D_{V}^{n - 2}}{p^{n + 1}}
    \right)
    &= 1
      - \frac{p - 1}{p^{2}} \left(\frac{p^{n - 1}}{2} - \frac{p - 1}{p}\right)
      - \frac{p - 1}{p} \\
    &= \frac{1}{2 p^{3}} \left(
      2p^{3} - (p - 1) (p^{n} - 2 (p - 1)) - 2 p^{2} (p - 1)\right) \\
    &= \frac{2 - p (p - 1)(p^{n - 1} - 4)}{2 p^{3}}.
  \end{align}
  Therefore, we see that the inequalities in \cref{prop:criterion:canonical} (2)
  hold when \(d \geq p - 1 + p^{n - 1}\).  Hence the quotient \(X = V/G\) is log
  canonical except when \((p, n) = (2, 3)\).  We remark that if
  \((p, n) = (2, 3)\), the \(G\)-representation \(V\) has pseudo-reflections.

  (1).  We consider the case \(d = p + p^{n - 1}\).  Similarly we have
  \begin{equation}
    1 - \frac{1}{p^{n - m}}
    - \sum_{l = m}^{n - 1} \frac{D_{V}^{(l)}}{p^{2n - 1 -l}}
    < 1 - \sum_{l = m}^{n - 1} \frac{D_{V}^{(l)}}{p^{2n - 1 -l}}
    \leq 1 - \frac{D_{V}^{(n - 1)}}{p^{n}}.
  \end{equation}
  By direct computing, we obtain
  \begin{gather}
    S_{d}^{(n - 1)} = p, D_{V}^{(n - 1)} = p^{n},
  \end{gather}
  and hence
  \begin{equation}
    1 - \frac{1}{p^{n - m}}
    - \sum_{l = m}^{n - 1} \frac{D_{V}^{(l)}}{p^{2n - 1 -l}}
    < 1 - \frac{p^{n}}{p^{n}} = 0.
  \end{equation}
  Therefore, then the quotient \(X = V/G\) is canonical if
  \(d \geq p + p^{n - 1}\).
\end{proof}

\section{Finite groups whose \texorpdfstring{\(p\)}{p}-Sylow subgroup is
  cyclic}\label{sec:application}
We can slightly generalize our main result as follows.

\begin{Theorem}\label{prop:application:p^n-Sylow}
  Let \(\tilde{G}\) be a finite group whose \(p\)-Sylow subgroup is
  \(G \cong \setZ / p^{n} \setZ\), and \(\tilde{V}\) a
  \(\tilde{G}\)-representation.  Assume that the restriction \(V\) of
  \(\tilde{V}\) to \(G\) has no pseudo-reflection.  Let
  \(\tilde{X} \coloneq \tilde{V}/\tilde{G}\) and \(X \coloneq V/ G\) be the
  quotient varieties.
  \begin{enumerate}
  \item \(\tilde{X}\) is log terminal if the inequalities
    \(1 - 1/p^{n - m} - \sum_{l = m}^{n - 1} D_{V}^{(l)}/p^{2n - 1 - l} < 0\)
    (\(m = 0, 1, \dotsc, n - 1\)) hold.
  \item \(\tilde{X}\) is log canonical if and only if the inequalities
    \(1 - 1/p^{n - m} - \sum_{l = m}^{n - 1} D_{V}^{(l)}/p^{2n - 1 - l} \leq 0\)
    (\(m = 0, 1, \dotsc, n - 1\)) hold.
  \end{enumerate}
\end{Theorem}
\begin{proof}
  Let \(\pi \colon X \to \tilde{X}\) be the canonical projection.

  (1).  If the inequalities hold, then \(X\) is canonical.  From
  \cite[Theorem~6.5]{Yamamoto2020:Pathological}, we see that \(\tilde{X}\) is
  log terminal.

  (2).  \(X\) is log canonical if and only if the inequalities hold.  If
  \(\tilde{X}\) is log canonical, from the contraposition of
  \cite[Theorem~6.4]{Yamamoto2020:Pathological}, \(X\) is log canonical.
  Conversely, with similar proof as in
  \cite[Theorem~6.5]{Yamamoto2020:Pathological}, we see that if \(X\) is log
  canonical, then \(\tilde{X}\) is log canonical.
\end{proof}

\begin{Remark}
  There are only finitely many indecomposable \(\tilde{G}\)-representations up
  to isomorphism.  Moreover, an explicit formula for the number of
  non-isomorphic ones is given in~\cite{Janusz1966:Indecomposable}.
\end{Remark}

\bibliography{ref}
\bibliographystyle{amsplain}
\end{document}